\documentclass[a4paper,10pt]{article}
\usepackage[latin1]{inputenc}
\usepackage[T1]{fontenc}
\usepackage{amssymb}
\usepackage{epsfig}
\usepackage{amssymb}
\usepackage{amsmath}
\usepackage{amsfonts}
\usepackage{graphicx, url}
\usepackage{makeidx}
\usepackage{cite}
\usepackage{color}
\usepackage{setspace}

\newtheorem{theorem}{Theorem}[section]

\newtheorem{lemma}[theorem]{Lemma}
\newtheorem{remark}[theorem]{Remark}

\newenvironment {proof} {{\it Proof.}}{\hspace*{\fill}$\Box$\par\vspace{4mm}}

\newcommand*{\QED}{\hfill\ensuremath{\square}}%

\newcommand{\mc}{\mathcal}
\newcommand{\mb}{\mathbb}

\title{A numerical method for solving stochastic differential equations with noisy memory}

\author{Kristina R. Dahl
\thanks{Department of Mathematics, University of Oslo. P.O box 1053, Blindern, 0316 OSLO, Norway. Email: kristrd@math.uio.no. Phone: +47-22854183}}

\begin{document}

\maketitle

\begin{abstract}
Stochastic differential equations with noisy memory are often impossible to solve analytically. Therefore, we derive a numerical Euler-Maruyama scheme for such equations and prove that the mean-square error of this scheme is of order $\sqrt{\Delta t}$. This is, perhaps somewhat surprisingly, the same order as the Euler-Maruyama scheme for regular SDEs, despite the added complexity from the noisy memory. To illustrate this numerical method, we apply it to a noisy memory SDE which can be solved analytically.
\end{abstract}

\bigskip


\medskip

\textbf{MSC classification codes:} 65C30 $\cdot$ 60H35 $\cdot$ 60H10 $\cdot$ 65C20

\section{Introduction}

In this paper, we study how noisy memory stochastic differential equations (SDEs), introduced in Dahl et al.~\cite{Dahl}, are connected to Volterra equations. We also discuss existence and uniqueness of solutions to noisy memory SDEs. Since such equations usually can not be solved analytically, we derive an Euler-Maruyama scheme for a numerical approximation of the solution. We prove that this scheme has mean square order of convergence $\sqrt{\Delta t}$.

One should note the following unique features of the analysis:

\begin{itemize}
\item{The stochastic differential equation (SDE) is driven by {\it generalized noisy memory}: The evolution of the state $X$ at any time $t$ is dependent on its past history
$\int_{t-\delta}^t \phi(t,s) X(s) \, dB(s)$ where $\delta$ is the memory span and $dB$ is white noise.}
\item{Noisy memory SDEs where the memory does not include a time-dependent function can be rewritten as two dimensional SDEs with delay (see Dahl et al.~\cite{Dahl}). Hence, one may solve such equations using numerical methods for delay SDEs, see e.g., Buckwar~\cite{Buckwar}, Carletti~\cite{Carletti}, Mao and Sabanis~\cite{MaoSabanis} and Milstein and Tretyakov~\cite{Milstein}. However, scaling the memory by a time-dependent function implies that generalized noisy memory SDEs cannot be rephrased as SDEs with delay. To the best of our knowledge, no current numerical methods work for approximating the solutions of generalized noisy memory SDEs. However, our numerical scheme works for all noisy memory SDEs, including the generalized ones.}
\item{We prove that the Euler-Maruyama scheme has mean square order of convergence $\sqrt{\Delta t}$. This is the same as the Euler-Maruyama method for classical SDEs. Hence, the added complexity from the noisy memory in the SDE does not reduce the order of convergence of the Euler-Maruyama scheme.}
\end{itemize}

Noisy memory SDEs can be applied to model animal populations where the population growth depends in some stochastic way on the previous population states, as well as the current number of animals. This kind of memory effect can be useful in the modeling of species where there is a natural delay in the population growth caused by e.g., hatching of eggs for fish, or larva becoming butterflies. This delay may depend on time, such as seasonal weather effects. This motivates generalized noisy memory SDEs. For applications of stochastic delay equations connected to population dynamics, see~\cite{Kuang}. Other applications of stochastic delay equations include spread of infectious diseases, see Beretta et al.~\cite{BerettaEtAl}, applications in physics and engeneering, see Kolmanovskii and Myshkis~\cite{Kolmanovskii} and financial applications, see {\O}ksendal and Sulem~\cite{OksendalSulem}.

Stochastic systems with memory, and the related field of stochastic systems with delay, has been an important field of research over the last years, see for example Mohammed and Zhang~\cite{MohammedZhang}, Mohammed~\cite{Mohammed_Memory} and {\O}ksendal, Sulem and Zhang~\cite{OSZ1}. Introducing a noisy memory $Z(t)$, as opposed to a perfect memory, is a natural generalization of this research. Verriest and Florchinger~\cite{Verriest}, Verriest~\cite{Verriest2} and Verriest and Michiels~\cite{VerriestMichiels} all consider stochastic delay differential equations and derive corresponding stability results on the solution. Li and Cao~\cite{Li} derive a two-step Euler-Maruyama method for a nonlinear neural stochastic delay differential equation.



The structure of the paper is as follows: In Section~\ref{sec: NoisySDEandEuler} we introduce the noisy memory SDE, show a connection between noisy memory SDEs and stochastic Volterra equations and give the Euler scheme to approximate the solution of the noisy SDE. Then, in Section~\ref{sec: main_result} we state our main result on the order of convergence of the Euler method and prove several lemmas needed to prove this. In Section~\ref{sec: proof_main}, we complete the proof of the main theorem. Finally, in Section~\ref{sec: num_ex}, we derive an analytical solution to a noisy SDE, and give a numerical example illustrating the convergence of the Euler method.

\section{The noisy memory SDE and the Euler scheme}
\label{sec: NoisySDEandEuler}

In this section, we introduce a stochastic differential equation with noisy memory, and derive the corresponding Euler scheme.

Let $B_t(\omega) = B(t, \omega); (t, \omega) \in [-\delta, \infty) \times \Omega$ be a Brownian motion on a complete filtered probability space $(\Omega, \mc{F}, \{\mc{F}_t\}_{t \geq 0}, P)$. We assume that $\mb{F} := \{\mc{F}_t\}_{t \geq 0}$ is the filtration generated by $\{B_t\}_{t \geq 0}$ (augmented with the $P$-null sets).

Let $\delta > 0$ be a time-delay determining the length of the memory-interval. Also, consider the functions $b : \Omega \times [0,T] \times \mb{R} \times \mb{R} \rightarrow \mb{R}$, $\sigma : \Omega \times [0,T] \times \mb{R} \times \mb{R} \rightarrow \mb{R}$, $\xi : [0,T] \rightarrow \mb{R}$ and $\phi :[0,T] \times [0,T] \rightarrow \mb{R}$. We will develop a numerical method for the following stochastic differential equation with noisy memory:

\begin{equation}
\label{eq: noisy_mem_SDE}
\begin{array}{llll}
dX(t) &=& b(t,X(t),Z(t))dt + \sigma(t,X(t),Z(t))dB(t), \mbox{ } t \in (0,T] \\[\smallskipamount]
X(t) &=& \xi(t), \mbox{ } t \in [-\delta,0],
\end{array}
\end{equation}
\noindent where the stochastic process
\begin{equation}
\label{eq: Z}
Z(t) := \int_{t-\delta}^t \phi(t, s)X(s) dB(s)
\end{equation}
\noindent is the (generalized) \emph{noisy memory} of $X(t)$, see also Dahl et al.~\cite{Dahl}. If $\phi(t,s)=1$ for all $t,s \in [0,T]$, equation~\eqref{eq: noisy_mem_SDE} is a non-generalized (or regular) noisy memory SDE. For more on stochastic differential equations in general, see for instance {\O}ksendal~\cite{Oksendal}. 

The parameter $\delta$ is the memory parameter which gives the length of the memory-interval. Note that the memory is noisy due to the It{\^o} integral in the definition. Intuitively, this means that the system does not have a perfect memory, but a slightly distorted one. The (deterministic) function $\phi$ inside the noisy memory It{\^o} integral allows this noisy memory to vary with time, both the time of the memory, but also the current time.

\begin{remark}
 \label{remark: Buckwar}
  \rm{Buckwar~\cite{Buckwar2} considers the same type of memory as in~\eqref{eq: Z}, but with a deterministic Lebesgue integral instead of a stochastic It{\^o} integral. That is, they consider a deterministic distributed memory instead of a stochastic distributed memory like us.} 
\end{remark}

Note that the Brownian motion is defined for negative times is $[-\delta,0]$. For a detailed presentation of how this is done, see Holden et al.~\cite{Holden} (Section 2.1.1).

\subsection{A connection between noisy memory SDEs and stochastic Volterra equations}
\label{sec: Volterra}

Consider a very simple SDE with noisy memory,

\begin{equation}
 \label{eq: simple_SDE_noisy}
 dX(t) = Z(t)dt = (\int_{t-\delta}^t X(u) dB(u)) dt
\end{equation}

\noindent where $Z$ is given as in equation~\eqref{eq: Z} with $\phi(t,s)=1$ for all $t,s \in [0,T]$. By changing the order of integration in \eqref{eq: simple_SDE_noisy}, we see that

\begin{equation}
\label{eq: noisySDE_Volterra_simple}
 \begin{array}{llll}
  X(t) &=& X(0) + \int_0^t \int_{s-\delta}^s X(u) dB(u) ds \\[\smallskipamount]
  &=& X(0) + \int_0^t \int_u^{\min\{u+\delta, t\}} 1 ds X(u) dB(u) \\[\smallskipamount]
  &=& X(0) +  \int_0^t \min\{t-u, \delta\} X(u) dB(u).
 \end{array}
\end{equation}

This is a linear stochastic Volterra equation, see {\O}ksendal and Zhang~\cite{OksendalZhang_Volterra}. Such equations do not have an simple analytical solution. However, they can be solved using an iterative method, see {\O}ksendal and Zhang~\cite{OksendalZhang_Volterra} and \cite{OksendalZhang_linear}.

Now, consider equation~\eqref{eq: noisy_mem_SDE}. If we assume that $b$ may be decomposed as $b(t,X(t),Z(t))= \tilde{b}(t,X(t)) + aZ(t)$, where $a \in \mb{R}$ and $\sigma(t,X(t),Z(t)) = \sigma(t,X(t))$, then we can rewrite equation~\eqref{eq: noisy_mem_SDE} as a stochastic Volterra equation:

\begin{equation}
\label{eq: noisy_mem_SDE_Volterra}
\begin{array}{llll}
X(t) &=& X(0) + \int_0^t b(s,X(s),Z(s))ds + \int_0^t\sigma(s,X(s))dB(s) \\[\smallskipamount]
&=& X(0) + \int_0^t \tilde{b}(s,X(s))ds + a\int_0^t \int_{s-\delta}^s \phi(s,u) X(u)dB(u)ds \\[\smallskipamount]
&&+ \int_0^t\sigma(s,X(s))dB(s) \\[\smallskipamount]
&=& X(0) + \int_0^t \tilde{b}(s,X(s))ds + a\int_0^t \int_{u}^{\min\{u+\delta,t \}} \phi(s,u) ds X(u)dB(u) \\[\smallskipamount]
&&+ \int_0^t\sigma(s,X(s))dB(s) \\[\smallskipamount]
&=& X(0) + \int_0^t \tilde{b}(s,X(s))ds + \int_0^t \{ \tilde{\phi}(t,s) X(s) + \sigma(s,X(s)) \}dB(s) \\[\smallskipamount]
\end{array}
\end{equation}
\noindent where the third equality follows from the same kind of calculations as in~\eqref{eq: noisySDE_Volterra_simple} and $\tilde{\phi}(t,s):= a \int_{s}^{\min\{ s+\delta,t \}} \phi(u,s) du$. This is a stochastic Volterra equation. Conditions for the existence of a unique solution to such equations can be found in e.g., Wang~\cite{Wang}.

The previous argument shows that non-trivial noisy memory SDEs are at least as difficult to solve as stochastic Volterra equations.

\subsection{Existence of solution}

In this section, we prove some results on the existence of a unique solution to equation~\eqref{eq: noisy_mem_SDE}.

\begin{theorem}
 \label{thm: existence_non_general}
In the (non-generalized) case where $\phi(t,s)=1$ for all $t,s \in [0,T]$, the following assumptions on the coefficient functions $b$ and $\sigma$ are sufficient for the existence of a unique solution to equation~\eqref{eq: noisy_mem_SDE}:
\begin{enumerate}
\item[$(i)$] The functions $ b(\omega, t,\cdot)$ and $\sigma(\omega,t, \cdot)$
are assumed to be $C^1$ for each fixed $\omega \in \Omega,t \in [0,T]$.

\item[$(ii)$] The functions  $b(\cdot,x,z)$ and $\sigma(\cdot, x,z)$ are predictable for each $x,z$.

\item[$(iii)$] \emph{Lipschitz condition:} The functions $b$ and $\sigma$ are Lipschitz continuous in the variables $x$ and $z$ with a Lipschitz constant $D$ which is independent of the variables $t, \omega$, i.e.:
\[
 \begin{array}{lll}
  |b(t,x_1,z_1)-b(t,x_2,z_2)| \leq D(|x_1-x_2| + |z_1-z_2|) \\[\smallskipamount]
  |\sigma(t,x_1,z_1)-\sigma(t,x_2,z_2)| \leq D(|x_1-x_2| + |z_1-z_2|).
 \end{array}
\]
\item[$(iv)$] \emph{Linear growth condition:} The functions $b$ and $\sigma$ satisfy the linear growth condition in the variables $x$ and $z$ with the linear growth constant $C$ independent of the variables $t,\omega$, i.e.:
\[
 \begin{array}{lll}
  |b(t,x,z)| + |\sigma(t,x,z)| \leq C(1 + |x| + |z|).
 \end{array}
\]
\end{enumerate}
 
\end{theorem}

\begin{proof}
Assumptions $(i)$ and $(ii)$ are sufficient to ensure that the integrands in equation~\eqref{eq: noisy_mem_SDE} have  predictable versions, whenever $X$ is c\`{a}dl\`{a}g and adapted. Together with the Lipschitz and linear growth conditions, this ensures that there exists a unique c\`{a}dl\`{a}g adapted solution $X$ to the equation~\eqref{eq: noisy_mem_SDE}, satisfying
\[
E[\sup_{t\in[-\delta,T]}|X(t)|^2]<\infty.
\]
This can be seen by regarding equation~\eqref{eq: noisy_mem_SDE} as a stochastic functional differential equation in the sense of Mohammed~\cite{MR754561}.
 
\end{proof}

However, we are also interested in having conditions for there to exist a unique solution to equation~\eqref{eq: noisy_mem_SDE} for some general function $\phi(t,s)$:

\begin{theorem}
 \label{thm: generalized}
Consider the generalized case, where $\phi(t,s)$ is some arbitrary function. Assume that $b(t,X(t),Z(t))= \tilde{b}(t,X(t)) + aZ(t)$, where $a \in \mb{R}$ and $\sigma(t,X(t),Z(t)) = \sigma(t,X(t))$. Then, under some fairly weak additional regularity conditions (see Wang~\cite{Wang}), there exists a unique solution to the noisy memory SDE~\eqref{eq: noisy_mem_SDE}. 
\end{theorem}

\begin{proof}
In this setting, the derivation of Section~\ref{sec: Volterra} combined with the conditions in Wang~\cite{Wang} guarantees existence of a unique solution.
\end{proof}

Note that it may be possible to prove the existence and uniqueness of a solution to equation~\eqref{eq: noisy_mem_SDE} in general (without assumptions on the functions $b$ and $\sigma$). However, this is beyond the scope of this paper. The Euler method presented here holds for all SDEs of the form \eqref{eq: noisy_mem_SDE}.

\subsection{The Euler scheme}

Let $\omega \in \Omega$ be a scenario. Let $N > 0$ be a (large) natural number and let the time step in the approximation, $\Delta t := T/N$. Then, $\Pi_{pos}:=\{n \Delta t\}_{n =0,1,\ldots,N}$ is a partition of the time interval $[0,T]$. Similarly, one can partition the interval $[-\delta,0]$ as $\Pi_{neg} :=\{-\delta, -\delta + \Delta t, ..., -\delta + k \Delta t\}$, where $k$ is the largest integer such that $-\delta + k \Delta t \leq 0$. For $i=1,\ldots, N$, let $\Pi_i$ denote the partition of the interval $[t_i-\delta, t_i]$ given by

\[
\Pi_i := (\Pi_{pos} \cup \Pi_{neg}) \cap [t_i-\delta,t_i],
\]
\noindent i.e., the partition of $[t_i-\delta,t_i]$ coming from the partition of the whole time interval.

A natural generalization of the Euler scheme for standard SDEs (see for instance Iacus~\cite{Iacus}) to the noisy memory SDE case is the following:
\begin{equation}
\label{eq: euler_scheme}
\begin{array}{lll}
X_{i+1}(\omega) = X_i(\omega) + b(t_i, X_i(\omega), Z_i(\omega)) \Delta t + \sigma(t_i,X_i(\omega),Z_i(\omega)) \Delta B_i(\omega)
\end{array}
\end{equation}
\noindent where $\Delta B_i(\omega) := B(t_{i+1},\omega) - B(t_{i},\omega)$ with distribution $N(0,\sqrt{\Delta t})$ are increments of the Brownian motion and $Z_i(\omega) := \sum_{j \in \Pi_i} \phi(t_i,t_j) X_j(\omega) \Delta B_j(\omega)$ approximates the noisy memory process. Note also that this is a pathwise (i.e., $\omega$-wise) approximation. However, in the next section, we will study the mean square error of the approximation in order to determine the convergence properties of this approximation to the exact solution.

Throughout the paper, we will assume that $\delta > \Delta t$. This assumption is valid, as we are interested in what happens for small time steps.

\section{The main result}
\label{sec: main_result}

It turns out that the noisy memory Euler scheme~\eqref{eq: euler_scheme} has mean-square order of convergence $\sqrt{\Delta t}$, which is the same as for ordinary SDEs (see Allen~\cite{Allen} and Mao~\cite{Mao}, Theorem 7.3). We summarize this in the following main result: 

\begin{theorem}
\label{thm: Euler_scheme}
The Euler approximation scheme for the solution of the stochastic noisy memory SDE \eqref{eq: noisy_mem_SDE} with constant time steps $\Delta t =\frac{T}{N}$ has mean-square order of convergence  $\sqrt{\Delta t}$. That is, there exists a constant $\tilde{C}(T)$ such that if $X$ is the exact solution of the noisy memory SDE and $X_i$ is the approximated solution (at the same point), then 

\[
E[(X(t_i) - X_i)^2] \leq \tilde{C}(T) \Delta t
\]

\noindent in all the approximation points $t_i$, $i=1, \ldots, N$.
\end{theorem}

The rest of this section is devoted to some lemmas which are needed to prove this theorem. The final proof of Theorem~\ref{thm: Euler_scheme} will be given in Section~\ref{sec: proof_main}.

\subsection{Some lemmas}

In this section we prove some lemmas concerning the solution of the noisy memory SDE, which will be used later on in order to compute the order of convergence for the Euler approximation scheme.

We need some Lipschitz-type conditions on the given functions. Assume that there exists constants $K_1, K_2 > 0$ (independent of $\omega \in \Omega$) such that
\begin{equation}
 \label{assumption: b_and_sigma}
 \begin{array}{llll}
  |b(t,x_1,z_1) - b(s,x_2,z_2)|^2 \leq K_1(|t-s| + |x_1-x_2|^2 + |z_1-z_2|^2) \\[\smallskipamount]
  |\sigma(t,x_1,z_1) - \sigma(s,x_2,z_2)|^2 \leq K_2(|t-s| + |x_1-x_2|^2 + |z_1-z_2|^2).
  \end{array}
\end{equation}

\noindent We also assume that there exists constants $K_3, K_4 > 0$ such that

\begin{equation}
 \label{assumption: b_and_sigma_2}
 \begin{array}{llll}
  b(t,x,z)^2 \leq K_3(1 + x^2 + z^2) \\[\smallskipamount]
  \sigma(t,x,z)^2 \leq K_4(1 + x^2 + z^2).
  \end{array}
\end{equation}

\noindent For notational simplicity, we let $k=\max\{K_1, K_2, K_3, K_4\}$. In addition, we assume that the (real valued, deterministic) function $\phi$ is square integrable, so there exists a constant $\tilde{K}$ such that

\begin{equation}
\label{eq: sq_int}
\int_{-\infty}^{\infty} \int_{-\infty}^{\infty} \phi(t,s)^2 dt ds \leq \tilde{K}.
\end{equation}

 %
 
 %

\noindent In the following, let $X$ be the solution of the noisy memory SDE~\eqref{eq: noisy_mem_SDE} and let $Z$ be the corresponding noisy memory process. In the following proofs, we will often use the inequality
\begin{equation}
\label{eq: quad}
2|ab| \leq a^2 + b^2.\hspace{0.5cm} 
\end{equation}
\noindent Note that inequality \eqref{eq: quad} implies that $(a+b)^2 \leq 2a^2 + 2b^2$.

\begin{lemma}
 \label{lemma: Xbounded}
 There exists a constant $M > 0$ such that $E[X(t)^2] \leq M$ for all $t \in [0,T]$.
\end{lemma}

\begin{proof}
 Define $v(t) := E[X(t)^2]$. Then,
 \[
  \begin{array}{lll}
 v(t) &= E[(X(0) + \int_0^t b(s,X(s),Z(s)) ds + \int_0^t \sigma(s,X(s),Z(s)) dB(s))^2] \\[\smallskipamount]
 &\leq 4E[X(0)^2] + 4E[(\int_0^t b(s,X(s),Z(s))ds)^2] + 2E[(\int_0^t \sigma(s,X(s),Z(s))dB(s))^2] \\[\smallskipamount]
 &=  4E[X(0)^2] + 4E[(\int_0^t b(s,X(s),Z(s))ds)^2] + 2E[(\int_0^t \sigma(s,X(s),Z(s))^2 ds] \\[\smallskipamount]
 &\leq 4E[X(0)^2] + 4t k \int_0^t E[1 + X(s)^2+ (\int_{s-\delta}^s \phi(s,u) X(u)dB(u))^2]ds \\[\smallskipamount]
 &\hspace{0.3cm} +2 k \int_0^t E[1 + X(s)^2+ (\int_{s-\delta}^s \phi(s,u) X(u)dB(u))^2]ds \\[\smallskipamount]
 &\leq 4E[X(0)^2] + 4t k \tilde{K} \int_0^t E[1 + X(s)^2+ \int_0^t X(u)^2 du]ds \\[\smallskipamount]
 &\hspace{0.3cm} +2 k \tilde{K} \int_0^t E[1 + X(s)^2+ \int_0^t X(u)^2du]ds \\[\smallskipamount]
 &= 2(2E[X(0)^2] + k \tilde{K} t (2t+1)) + 2 k \tilde{K} (2t+1)(1+t) \int_0^t v(s)ds
 \end{array}
 \]

\noindent where the first inequality us inequality \eqref{eq: quad} twice, the second equality uses the It{\^o} isometry, the second inequality uses the Cauchy-Schwarz inequality and assumption \eqref{assumption: b_and_sigma_2}, the third equality uses the It{\^o} isometry. Hence,
\[
v(t) \leq a(t) + b(t)\int_0^t v(s)ds,
\]
\noindent where $a(t), b(t)$ are real-valued functions given by

\[
\begin{array}{lll}
 a(t):=2(2E[X(0)^2] + k \tilde{K} t (2t+1)) \\[\smallskipamount]
 b(t) := 2 k \tilde{K} (2t+1)(1+t).
\end{array}
\]

By Gr{\"o}nwall's inequality (see {\O}ksendal~\cite{Oksendal}), this implies that

\[
\begin{array}{lll}
 E[X(t)^2] 
 &\leq a(T) + b(T) \int_0^T e^{b(T-s)}ds =:M,
 \end{array}
\]
\noindent where the second inequality uses that $k > 0$ and

\[
 M :=  a(T) + b(T)\frac{e^{2 k \tilde{K} (3T + 2 T^2 +1)}}{2 k \tilde{K} (4T+3)} - \frac{e^{2 k \tilde{K}}}{6 k \tilde{K}}.
\]

This proves that $E[X(t)^2]$ is bounded.

\end{proof}

\begin{lemma}
 \label{lemma: X_lipschitz}
 There exists a constant $c > 0$ such that for all $t,s \in [0,T]$,

 \[
  E[(X(t) - X(s))^2] \leq c|t-s|.
 \]

\end{lemma}

\begin{proof}
Let $t,s \in [0,T]$, $s \leq t$ (if $t < s$, change roles),
\[
\begin{array}{lll}
 E[(X(t)-X(s))^2] &= E[(\int_s^t b(u,X(u), Z(u))du + \int_s^t \sigma(u,X(u), Z(u))dB(u))^2] \\[\smallskipamount]
 &\leq 2E[(\int_s^t b(u,X(u), Z(u))du)^2] + 2E[(\int_s^t \sigma(u,X(u), Z(u))dB(u))^2] \\[\smallskipamount]
 &\leq 2(t-s) \int_s^t E[b(u,X(u),Z(u))^2] du + 2\int_s^tE[\sigma(u,X(u),Z(u))]du \\[\smallskipamount]
 &\leq 2(t-s) k \int_s^t (1 + E[X(u)^2] + E[(\int_{u-\delta}^u \phi(u,w) X(w) dB(w))^2])du \\[\smallskipamount]
 &\hspace{0.3cm} + 2 k \int_s^t (1 + E[X(u)^2] + E[(\int_{u-\delta}^u \phi(u,w) X(w) dB(w))^2]) du \\[\smallskipamount]
 &= 2(t-s) k \tilde{K} \int_s^t (1 + E[X(u)^2] + E[\int_{u-\delta}^u X(w)^2 dw])du \\[\smallskipamount]
 &\hspace{0.3cm}+ 2 k \tilde{K} \int_s^t (1 + E[X(u)^2] + E[\int_{u-\delta}^u X(w)^2 dw]) du \\[\smallskipamount]
 &\leq 2(t-s) k \tilde{K} \int_s^t (1 + M + \int_{u-\delta}^u M dw)du \\[\smallskipamount]
 &\hspace{0.3cm}+ 2 k \tilde{K} \int_s^t (1 + M + \int_{u-\delta}^u M dw) du \\[\smallskipamount]
 &= 2 k \tilde{K} (t-s)^2(1 + M + M \delta) + 2 k \tilde{K}(1 + M + M \delta)(t-s) \\[\smallskipamount]
 &\leq c (t-s)
\end{array}
\]
\noindent where $c > 0$ is a constant and the first inequality uses some algebra and inequality \eqref{eq: quad}, the second inequality uses the It{\^o} isometry and the Cauchy-Schwartz inequality, the third inequality follows from assumption~\eqref{assumption: b_and_sigma_2}, the  second equality follows from the It{\^o} isometry and the fourth inequality follows from Lemma~\ref{lemma: Xbounded}. Note that the final inequality holds since $(t-s) \leq T$.
\end{proof}

\begin{lemma}
\label{lemma: Z_bounded}
It holds that $E[Z(t)^2] \leq \tilde{K} M\delta$ for all $t \in [0,T]$.
\end{lemma}

\begin{proof}
 \[
 \begin{array}{lll}
E[Z(t)^2]  &=& E[(\int_{t-\delta}^t \phi(t,s) X(s) dB(s))^2] \\[\smallskipamount]
&=& E[\int_{t-\delta}^t \phi(t,s)^2 X(s)^2 ds] \\[\smallskipamount]
&\leq& \int_{t-\delta}^t \tilde{K} M ds = \tilde{K} M \delta,
\end{array}
 \]
\noindent where the second equality uses the It{\^o} isometry (see e.g., {\O}ksendal~\cite{Oksendal}) and the inequality follows from Lemma~\ref{lemma: Xbounded}.
\end{proof}

\begin{lemma}
 \label{lemma: Z_lipschitz}
 There exists a constant $\tilde{N} > 0$ such that
 \[
  E[|Z(t)-Z(s)|^2] \leq \tilde{N}|t-s|.
 \]

\end{lemma}

\begin{proof}
Assume that $t > s$. If not, change the roles of $t$ and $s$.

We consider two cases:

 \begin{enumerate}
  \item[$(i)$] $s \notin [t-\delta,t]$, i.e., $s < t-\delta$: Then,
  \[
  \begin{array}{llll}
	E[|Z(t)-Z(s)|^2] = E[|\int_{t-\delta}^t \phi(t,u)X(u) dB(u) - \int_{s-\delta}^s \phi(s,u)X(u) dB(u)|^2] \\[\smallskipamount]
	\hspace{1.7cm} \leq 2E[(\int_{t-\delta}^t \phi(t,u) X(u) dB(u))^2] + 2E[(\int_{s-\delta}^s \phi(s,u) X(u) dB(u))^2] \\[\smallskipamount]
	\hspace{1.7cm} = 2E[\int_{t-\delta}^t \phi(t,u)^2 X(u)^2 du] + 2E[\int_{s-\delta}^s \phi(s,u)^2 X(u)^2 du] \\[\smallskipamount]
	\hspace{1.7cm} \leq 2 \int_{t-\delta}^t \tilde{K}M du + 2\int_{s-\delta}^s \tilde{K}M du 
	\leq 4M\tilde{K} (t-s)
   \end{array}
   \]
\noindent where the first inequality uses inequality \eqref{eq: quad}, the second equality uses the It{\^o} isometry, the second inequality uses Lemma~\ref{lemma: Xbounded} and the final inequality follows from $s < t-\delta$, i.e., $\delta < t-s$.

\item[$(ii)$] $s \in [t-\delta,t]$: In this case,
\[
  \begin{array}{llll}
	E[|Z(t)-Z(s)|^2] = E[|\int_{t-\delta}^t \phi(t,u) X(u) dB(u) - \int_{s-\delta}^s \phi(s,u) X(u) dB(u)|^2] \\[\smallskipamount]
	\hspace{1.3cm} =E[|-\int_{s-\delta}^{t-\delta} \phi(s,u)X(u)dB(u) + \int_{t-\delta}^s (\phi(t,u) - \phi(s,u)) X(u)dB(u) \\[\smallskipamount]
	\hspace{1.7cm} + \int_s^t \phi(t,u)X(u) dB(u)|^2] \\[\smallskipamount]
	\hspace{1.3cm} \leq 2E[|\int_{t-\delta}^{s} (\phi(t,u) - \phi(s,u))X(u)dB(u)|^2] \\[\smallskipamount]
	\hspace{1.7cm} + 2E[|\int_s^t\phi(t,u) X(u)dB(u) - \int_{s-\delta}^{t-\delta} \phi(s,u) X(u) dB(u)|^2] \\[\smallskipamount]
	\hspace{1.3cm} \leq 4 E[(\int_s^t \phi(t,u)X(u)dB(u))^2] + 4E[(\int_{s-\delta}^{t-\delta} \phi(s,u)X(u) dB(u))^2] \\[\smallskipamount]
	\hspace{1.7cm} +  2\int_{t-\delta}^{s} (\phi(t,u) - \phi(s,u))^2 E[X(u)^2]du \\[\smallskipamount]
	\hspace{1.3cm} \leq 4\int_s^t \tilde{K} E[X(u)^2]du + 4\int_{s-\delta}^{t-\delta} \tilde{K} E[X(u)^2]du + 2\int_{t-\delta}^s \tilde{K}(t-s)M du \\[\smallskipamount]
	\hspace{1.3cm}\leq 2(t-s)\tilde{K}M(4+ \delta)
   \end{array}
   \]
\noindent where the second equality follows from $s \in [t-\delta,t]$, the first and second inequality follows from inequality \eqref{eq: quad}, the third inequality follows from the It{\^o} isometry and assumptions~\eqref{eq: phi_begrenset}-\eqref{eq: phi_lin}, the final inequality follows from Lemma~\ref{lemma: Xbounded}.
\end{enumerate}

Hence, by combining the two items above, we see that

\[
 E[|Z(t)-Z(s)|^2] \leq \max\{2M\tilde{K} (4+ \delta), 4M\tilde{K}\}(t-s).
\]

The lemma follows by defining $\tilde{N}$ to be this maximum.
\end{proof}

\section{Error analysis and proof of the main theorem}
\label{sec: proof_main}

In this section, we derive an error bound for the Euler approximation method for SDEs with generalized noisy memory. We shall see that the approximation converges to the solution of the noisy memory SDE and find the order of convergence, and thereby prove our main result, Theorem~\ref{thm: Euler_scheme}.

Similarly to Allen~\cite{Allen}, for $t \in [t_i,t_{i+1}]$, $i=1, \ldots, N$, define

\begin{equation}
\label{eq: X_tilde}
\hat{X}(t) := X_i + \int_{t_i}^t b(t_i, X_i, Z_i)ds + \int_{t_i}^t \sigma(t_i,X_i(\omega),Z_i(\omega)) dB(s).
\end{equation}

Note that $\hat{X}(t_i) = X_i$ for $i=1,2,\ldots,N$, i.e., in the time nodes, the process $\hat{X}$ equals the approximation to the solution of the noisy memory process.

We study the error

\begin{equation}
 \label{eq: error}
 \epsilon(t) = X(t) - \hat{X}(t)
\end{equation}
\noindent where $X$ is the exact solution to the noisy memory SDE~\eqref{eq: noisy_mem_SDE}. The goal of this section is to prove that there exists a constant $\tilde{C}$ such that $E[\epsilon(t_i)^2] \leq \tilde{C} \Delta t$ for $t_i \in \Pi_{pos}$.

From the definitions, 

\[
 d\epsilon(t) = (b(t,X(t),Z(t)) - b(t_i, X_i, Z_i)) dt + (\sigma(t,X(t),Z(t)) - \sigma(t_i,X_i,Z_i)) dB(t)
\]
\noindent and $\epsilon(t_i)=X(t_i)-X_i$ for $i=1, 2, \ldots, N$.

From It{\^o}'s formula applied to the function $g(t,\epsilon)=\epsilon^2$, we see that

\[
 \begin{array}{llll}
  d[\epsilon(t)^2] &=& 2 (X(t) - \hat{X}(t))(b(t,X(t),Z(t)) - b(t_i, X_i,Z_i)) dt \\[\smallskipamount]
  &&+ 2(X(t)-\hat{X}(t))(\sigma(t,X(t),Z(t)) - \sigma(t_i, X_i, Z_i)) dB(t) \\[\smallskipamount]
  &&+ (\sigma(t,X(t),Z(t)) - \sigma(t_i,X_i,Z_i))^2 dt.
 \end{array}
\]

Hence,

\begin{equation}
\label{eq: stjerne}
\begin{array}{lll}
 E[\epsilon(t_{i+1})^2] &=& E[\epsilon(t_i)^2] + 2E[\int_{t_i}^{t_{i+1}} (X(s) - \hat{X}(s)) (b(s,X(s),Z(s)) - b(t_i,X_i,Z_i)) dt] \\[\smallskipamount]
 &&+ E[\int_{t_i}^{t_{i+1}} (\sigma(t,X(t),Z(t)) - \sigma(t_i,X_i,Z_i))^2 dt] \\[\smallskipamount]
 &\leq& E[\epsilon(t_i)^2] + \int_{t_i}^{t_{i+1}} E[(X(s) - \hat{X}(s))^2] dt \\[\smallskipamount]
 &&+ \int_{t_i}^{t_{i+1}} E[(b(s,X(s),Z(s)) - b(t_i,X_i,Z_i))^2] dt \\[\smallskipamount]
 &&+ \int_{t_i}^{t_{i+1}} E[(\sigma(t,X(t),Z(t)) - \sigma(t_i,X_i,Z_i))^2] dt
\end{array}
\end{equation}
\noindent where the inequality follows from inequality \eqref{eq: quad}.

Note that
\[
\begin{array}{lll}
 E[(b(t,X(t),Z(t)) - b(t_i,X_i,Z_i))^2] \\[\smallskipamount]
 \hspace{0.7cm}= E[(b(t,X(t),Z(t)) - b(t_i,X(t_i),Z(t_i)) + b(t_i,X(t_i),Z(t_i))- b(t_i,X_i,Z_i))^2] \\[\medskipamount]
 \hspace{0.7cm}\leq 2E[(b(t,X(t),Z(t)) - b(t_i,X(t_i),Z(t_i))^2] + 2E[(b(t_i,X(t_i),Z(t_i))- b(t_i,X_i,Z_i))^2] \\[\medskipamount]
 \hspace{0.7cm}\leq 2k E[|t-t_i| + |X(t) - X(t_i)|^2 + |Z(t) - Z(t_i)|^2 + |X(t_i) - X_i|^2 + |Z(t_i) - Z_i|^2]
\end{array}
 \]
\noindent where the first inequality follows from the triangle inequality and inequality \eqref{eq: quad}. The final inequality follows from the assumption~\eqref{assumption: b_and_sigma}. Similarly, one can prove that
\[
\begin{array}{lll}
 E[(\sigma(t,X(t),Z(t)) - \sigma(t_i,X_i,Z_i))^2] &\leq& 2k E[|t-t_i| + |X(t) - X(t_i)|^2 \\[\smallskipamount]
 &&+ |Z(t) - Z(t_i)|^2 + |X(t_i) - X_i|^2 + |Z(t_i) - Z_i|^2].
\end{array}
 \]
Therefore, combining this with inequality~\eqref{eq: stjerne} and using the definition of the error $\epsilon(t)$,
\begin{equation}
\label{eq: error_X_mellom}
\begin{array}{lll}
  E[\epsilon(t_{i+1})^2] &\leq& E[\epsilon(t_i)^2] + \int_{t_i}^{t_{i+1}} \epsilon(t) dt + 4k \int_{t_i}^{t_{i+1}} E[|t-t_i| + |X(t) - X(t_i)|^2 \\[\smallskipamount]
 &&+ |Z(t) - Z(t_i)|^2 + |X(t_i) - X_i|^2 + |Z(t_i) - Z_i|^2] dt.
 \end{array}
\end{equation}

Due to the noisy memory process, there is an additional source of error, compared to approximation of regular SDEs. In the following, let $X(t)$ be an exact solution of the noisy memory SDE ~\eqref{eq: noisy_mem_SDE}, and let $X_j$, $t_j \in [0,T]$ be its approximation from the Euler method~\eqref{eq: euler_scheme}. For $i=1, \ldots, N$, define $Z_i^B := \sum_{j \in \Pi_i} X(t_j) \Delta B_j$, i.e., the approximated noisy memory process involving the exact solution $X$.

\begin{lemma}
 \label{lemma: Z_minus_Z_B}
 For a time $t_i$ in the partition of the time interval and $Z_i^B = \sum_{j \in \Pi_i} \phi(t_i,t_j)X(t_j) \Delta B_j$, we have
 \[
  E[|Z(t_i) - Z_i^B|^2] \leq \Delta t \delta \tilde{K}(M +  c).
 \]
\end{lemma}

\begin{proof}
From the definitions,
 \[
  \begin{array}{lll}
   E[|Z(t_i) - Z_i^B|^2] &=& E[|\int_{t_i-\delta}^{t_i} \phi(t_i,s) X(s) dB(s) - \sum_{j \in \Pi_i} \phi(t_i,t_j)X(t_j) \Delta B_j|^2] \\[\smallskipamount]
   &=& E[|\sum_{j \in \Pi_i} \int_{t_j}^{t_{j+1}} (\phi(t_i,s)X(s) - \phi(t_i,t_j)X(t_j))dB(s) |^2] \\[\smallskipamount]
   &=& E[|\sum_{j \in \Pi_i} \int_{t_j}^{t_{j+1}} ( \{\phi(t_i,s)X(s) - \phi(t_i,t_j)X(s) \} \\[\smallskipamount]
   &&+ \{ \phi(t_i,t_j)X(s)- \phi(t_i,t_j)X(t_j)\} )dB(s) |^2] \\[\smallskipamount]
   &=& 2E[\sum_{j \in \Pi_i} \int_{t_j}^{t_{j+1}} X(s)^2(\phi(t_i,s) - \phi(t_i,t_j))^2 ds] \\[\smallskipamount]
   &&+2E[\sum_{j \in \Pi_i} \int_{t_j}^{t_{j+1}} (X(s)-X(t_j))^2 \phi(t_i,t_j)^2 ds] \\[\smallskipamount]
   &\leq& 2 \sum_{j \in \Pi_i} \int_{t_j}^{t_{j+1}} M \tilde{K} (s-t_j) ds + 2 \sum_{j \in \Pi_i} \int_{t_j}^{t_{j+1}} c(s - t_j) \tilde{K} ds \\[\smallskipamount]
   &=& \delta \Delta t \tilde{K}(M  + c )
  \end{array}
 \]
 \noindent where the fourth equality follows from the It\^{o} isometry (see e.g. {\O}ksendal \cite{Oksendal}) and the inequality from Lemma~\ref{lemma: Xbounded}, Lemma~\ref{lemma: X_lipschitz} and assumptions~\eqref{eq: phi_begrenset}-\eqref{eq: phi_lin}. 
 \end{proof}

We can now prove the following lemma which relates the error in the noisy memory process $Z$ to the error in the solution process $X$.

\begin{lemma}
 \label{lemma: error_Z}
 For $i=1, \ldots, N$, let $Z(t_i)$ be the noisy memory process, and $Z_i$ the approximated noisy memory process, then
 \[
  E[|Z(t_i)-Z_i|^2] \leq 2\Delta t \delta \tilde{K}(M + c)  + \tilde{K} \Delta t \sum_{j \in \Pi_i} E[\epsilon(t_j)^2].
 \]

\end{lemma}

\begin{proof}
First, note that

\[
 \begin{array}{lll}
  E[|Z_i^B-Z_i|^2] &=& E[(\sum_{j \in \Pi_i} X(t_j) \phi(t_i,t_j) \Delta B_j - \sum_{j \in \Pi_i} X_j \phi(t_i,t_j) \Delta B_j)^2] \\[\smallskipamount]
  &=& \sum_{j \in \Pi_i} \phi(t_i,t_j)^2 E[(X(t_j) - X_j)^2] \Delta t \\[\smallskipamount]
  &\leq& \tilde{K} \Delta t \sum_{j \in \Pi_i} E[\epsilon(t_j)^2]
 \end{array}
\]
\noindent where the final equality uses the discrete It{\^o} isometry. Therefore,

 \[
  \begin{array}{lll}
   E[|Z(t_i)-Z_i|^2] &=& E[|Z(t_i) - Z_i^B + Z_i^B-Z_i|^2] \\[\smallskipamount]
   &\leq& 2E[|Z(t_i) - Z_i^B|^2] + 2E[|Z_i^B-Z_i|^2] \\[\smallskipamount]
   &\leq& 2\Delta t \delta \tilde{K}  (M  + c)  + \tilde{K} \Delta t \sum_{j \in \Pi_i} E[\epsilon(t_j)^2]
  \end{array}
 \]
\noindent where the first inequality uses inequality \eqref{eq: quad} and the second inequality uses Lemma~\ref{lemma: Z_minus_Z_B}.

\end{proof}

We are nearly ready to prove our main result, Theorem~\ref{thm: Euler_scheme}. However, we need one more lemma:

\begin{lemma}
\label{lemma: exponent}
Let $x >0$ and $n \in \mb{N}$. Then,

\[
(1+\frac{x}{n})^n \leq e^x.
\]
\end{lemma}

\begin{proof}
The exponential function is convex, and therefore it dominates its first order Taylor approximation at $0$, so $e^y \geq 1+y$ for all $y$. By insering $y=x/n$ and taking the $n$'th power, the desired inequality follows.







\end{proof}

\noindent Finally, using all of these lemmas, we are ready to prove our main result.

\medskip

\noindent \textbf{Proof of Theorem~\ref{thm: generalized}:}

Recall from Theorem~\ref{thm: Euler_scheme} that we would like to prove that the expected squared error of the numerical scheme is bounded by some constant (depending on the terminal time) times the time step. That is, we want to prove that $E[\epsilon(t_{i})^2] \leq \tilde{C}(T) \Delta t$. By combining inequality \eqref{eq: error_X_mellom} with Lemma~\ref{lemma: X_lipschitz}, Lemma~\ref{lemma: Z_lipschitz} and Lemma~\ref{lemma: error_Z}, we see that

\[
\begin{array}{lll}
 E[\epsilon(t_{i+1})^2] &\leq& E[\epsilon(t_i)^2] + \int_{t_i}^{t_{i+1}} E[\epsilon(t)^2]dt + 4kE[\epsilon(t_i)^2]\Delta t + 2k(\Delta t)^2 \\[\smallskipamount]
 &&+ 4k c \int_{t_i}^{t_{i+1}} (t-t_i)dt + 4k\int_{t_i}^{t_{i+1}} \tilde{N}(t-t_i)dt \\[\smallskipamount]
 &&+ 4k\int_{t_i}^{t_{i+1}}(2\Delta t \delta (M\tilde{K} + \tilde{K}c )  + \tilde{K} \Delta t \sum_{j \in \Pi_i} E[\epsilon(t_j)^2])dt \\[\smallskipamount]
 &=&E[\epsilon(t_i)^2](1+  4k \Delta t) + 2k(\Delta t)^2(1  + c + \tilde{N}+ 2(2 \delta (M \tilde{K} + \tilde{K} c)  \\[\smallskipamount]
 &&+ \tilde{K} \sum_{j \in \Pi_i} E[\epsilon(t_j)^2]) + \int_{t_i}^{t_{i+1}}E[\epsilon(t)^2] dt.
 \end{array}
\]

Hence, by using the Bellman-Gr{\"o}nwall inequality, we see that

\[
\begin{array}{lll}
 E[\epsilon(t_{i+1})^2] &\leq& E[\epsilon(t_i)^2](1+  4k \Delta t) + 2k(\Delta t)^2(1  + c + \tilde{N}+ 2(2 \delta (M \tilde{K} + \tilde{K} c)\\[\smallskipamount]
 &&+ \tilde{K}  \sum_{j \in \Pi_i} E[\epsilon(t_j)^2]) + \int_{t_i}^{t_{i+1}} e^{t_{i+1}-t} \{E[\epsilon(t_i)^2](1+  4k \Delta t) \\[\smallskipamount]
 &&+ 2k(\Delta t)^2(1  + \tilde{N} + c+ 2(2 \delta (M\tilde{K} + \tilde{K}c)+ \tilde{K} \sum_{j \in \Pi_i} E[\epsilon(t_j)^2])\} dt \\[\smallskipamount]
 &=& e^{\Delta t}\{E[\epsilon(t_i)^2](1+  4k \Delta t) + 2k(\Delta t)^2(1  + c + \tilde{N}+ 4 \delta (M\tilde{K} + \tilde{K} c)\\[\smallskipamount]
 &&+ 2\tilde{K}  \sum_{j \in \Pi_i} E[\epsilon(t_j)^2]) \}
\end{array}
 \]

For $i=1, \ldots,N$, define $a_i := E[\epsilon(t_i)^2]$. From the previous computations, we know that

\begin{equation}
\label{eq: a_i}
a_{i+1} \leq R a_i + S \sum_{j \in \Pi_i} a_j + A
\end{equation}

\noindent where
\[
R:=e^{\Delta t}(1 + 4k\Delta t) > 0, \mbox{ } S:= 4 \tilde{K} k (\Delta t)^2 e^{\Delta t} > 0
 \]

 \noindent and
 \[
 A:= 2k(\Delta t)^2 e^{\Delta t}(1  + c + \tilde{N} + 4 \delta (M\tilde{K} + \tilde{K} c)) > 0.
 \]

Note that,

\begin{equation}
\label{eq: for_indusksjon}
\begin{array}{lll}
a_{i+1} &\leq& Ra_i + A + S\sum_{j \in \Pi_i} a_j \\[\smallskipamount]
&\leq& R \max_{j \in \Pi_i} a_j + S \sum_{j \in \Pi_i} \max_{j \in \Pi_i} a_j + A \\[\smallskipamount]
&=& (R + S \frac{\delta}{\Delta t})\max_{j \in \Pi_i} a_j +A \\[\smallskipamount]
&=:& \tilde{R} \max_{j \in \Pi_i} a_j +A
\end{array}
\end{equation}
\noindent where $\tilde{R} := R +  S \frac{\delta}{\Delta t} = e^{\Delta t} (1 + 4k \Delta t (1+\delta \tilde{K}))$.

By induction, and the fact that the initial approximation error is $0$, inequality~\eqref{eq: for_indusksjon} implies that $a_n \leq \tilde{A} \frac{\tilde{R}^n -1}{\tilde{R} - 1}$ for $n=1, \ldots, N$, i.e.,

\begin{equation}
\label{eq: nesten_ferdig}
\begin{array}{lll}
E[\epsilon(t_n)^2] &\leq& 2e^{\Delta t} k (\Delta t)^2 (1 + c + \tilde{N} + 4\delta (M\tilde{K} + \tilde{K} c)) \frac{e^{N\Delta t}(1 +4k(1 +\delta \tilde{K}) \Delta t)^n - 1}{e^{\Delta t}(1 + 4k(1 +\delta \tilde{K})  \Delta t) -1} \\[\smallskipamount]
&\leq& \Delta t \frac{e^T}{2(1+\delta \tilde{K})}(1 + 4k (1 +\delta \tilde{K}) \Delta t)^n (1 + c + k + 4 \delta (M \tilde{K} + \tilde{K} c)).
\end{array}
\end{equation}

Now, note that $e^{4k(1 +\delta \tilde{K})T} \geq (1+\frac{4k(1 +\delta \tilde{K})T}{n})^n$ because of Lemma~\ref{lemma: exponent}. By combining Lemma~\ref{lemma: exponent} with the inequality~\eqref{eq: nesten_ferdig} and recalling that $\Delta t = \frac{T}{N}$, we reach our goal

\[
\begin{array}{lll}
E[\epsilon(t_n)^2] &\leq& \Delta t \frac{e^{T(1+4k(1 +\delta \tilde{K}) )}}{2 (1+\delta \tilde{K})}(1 + c + \tilde{N} + 4\delta (M \tilde{K} + \tilde{K} c)) \\[\smallskipamount]
&=:& \Delta t \tilde{C}(T)
\end{array}
\]
\noindent where $\tilde{C}(T):=\frac{e^{T(1+4k(1 +\delta \tilde{K}) )}}{2(1+\delta \tilde{K})}(1 + c + \tilde{N} + 4\delta (M \tilde{K} + \tilde{K} c))$. This completes the proof of Theorem~\ref{thm: Euler_scheme}. \QED

\section{A noisy memory SDE with an analytical solution and a numerical example}
\label{sec: num_ex}

In this section, we will compare the exact solution of a (very simple) SDE with noisy memory to the approximation given by the Euler method. We consider the following SDE with noisy memory:

\begin{equation}
\label{eq: num_ex_simple_SDE}
\begin{array}{lll}
dX(t) &=& Z(t)dB(t) \mbox{ for } t \in [0,T] \\[\smallskipamount]
X(t) &=& 1, \mbox{ } t \in [-\delta,0).
\end{array}
\end{equation}

\noindent where $Z(t) := \int_{t-\delta}^t X(s) dB(s)$, so $\phi(t,s)=1$ for all $t,s \in [0,T]$. We can solve \eqref{eq: num_ex_simple_SDE} analytically by using a technique from Dahl et al. \cite{Dahl}, based on rewriting the noisy SDE~\eqref{eq: num_ex_simple_SDE} as a two-dimensional SDE with delay. This kind of delay equation can be solved iteratively for each $\delta$-interval.

First, we rewrite the noisy SDE~\eqref{eq: num_ex_simple_SDE} by defining $X_1(t) :=X(t)$ and $X_2(t):=\int_{-\delta}^t X_1(s)dB(s)$,  $t \in [-\delta,T]$. Note that from these definitions, $Z(t) = X_2(t)-X_2(t-\delta)$, $t \in [-\delta,T]$. Then, the noisy SDE~\eqref{eq: num_ex_simple_SDE} can be rewritten as a two-dimensional SDE with delay:

\begin{equation}
\label{eq: rewritten_SDE}
\begin{array}{llll}
dX_1(t)&=& (X_2(t)-X_2(t-\delta))dB(t), t \in (0,T], \\[\smallskipamount]
dX_2(t)&=&X_1(t)dB(t), t \in (-\delta,T], \\[\smallskipamount]
X_1(t)&=&1, t \in [-\delta,0], \\[\smallskipamount]
X_2(-\delta)=0.
\end{array}
\end{equation}

Note that $X_1(t)$ and $X_2(t)$ are known from the initial conditions for $t \in [-\delta,0]$. We write \eqref{eq: rewritten_SDE} in matrix form. Define $Y(t) := (X_1(t),  X_2(t))^{\intercal}$ (where $(\cdot,\cdot)^{\intercal}$ denotes the transpose), $t \in [-\delta,T]$. Then, from \eqref{eq: rewritten_SDE}

\begin{equation}
\label{eq: rewritten_matrix}
\begin{array}{lll}
dY(t) &=& \textbf{a} Y(t) dB(t) + \textbf{b} Y(t-\delta)dB(t), t \in [0,T]\\[\smallskipamount]
Y(t) &=& (1, B(t)-B(-\delta))^{\intercal}, t \in [-\delta,0],
\end{array}
\end{equation}

\noindent where $\textbf{a}$ and $\textbf{b}$ are in $\mb{R}^{2 \times 2}$ and defined by

\[
\textbf{a}=  \begin{bmatrix}
        0 & 1 \\ 1 & 0
    \end{bmatrix}, \mbox{ }
\textbf{b}=  \begin{bmatrix}
        0 & -1 \\ 0 & 0
    \end{bmatrix}, \mbox{ }
\]

For $t \in [0,\delta]$, we may rewrite \eqref{eq: rewritten_matrix} as

\[
\begin{array}{lll}
dY(t) &=& \textbf{a} Y(t)dB(t) + \begin{bmatrix}  B(-\delta)-B(t-\delta) \\ 0 \end{bmatrix} dB(t),
\end{array}
\]

\noindent which is a regular SDE without delay. For notational simplicity, define $K(t-\delta):=B(t-\delta)-B(-\delta)$. Note that for $t \in [0,\delta]$, $K(t-\delta)$ is a known process which is independent of everything after time $0$.

To solve this equation, define

\[
F(t) := \exp(-\textbf{a}B(t) + \frac{1}{2}\textbf{a}^2 t)
\]
\noindent where we (in general) define the matrix exponential for a matrix $\textbf{A} \in \mb{R}^{n \times n}$, $n \in \mb{N}$, as

\[
\exp(\textbf{A}) = \sum_{n=0}^{\infty} \frac{1}{n!} \textbf{A}^n.
\]

By this definition, we find (by analyzing the infinite sum) that

\[
\begin{array}{lll}
 F(t)=e^{\frac{1}{2}t} \begin{bmatrix} \cosh B(t) & -\sinh B(t) \\ -\sinh B(t) & \cosh B(t) \end{bmatrix}.
\end{array}
\]

Note also that
\[
\textbf{a}^2=\begin{bmatrix} 1 & 0 \\ 0 & 1 \end{bmatrix},
\]
\noindent which clearly commutes with the matrices $F(t)$ (for all $t$) and $\textbf{a}$. Also, the matrices $\textbf{a}$ and $F(t)$ commute (for all $t$). This justifies the following calculations:

By the two-dimensional It{\^o} formula,

\[
\begin{array}{lll}
dF(t) &=& F(t)(\textbf{a}^2 dt - \textbf{a} dB(t)).
\end{array}
\]

Hence, by the It{\^o} product rule,

\[
d(F(t)Y(t))= F(t)\begin{bmatrix} -K(t-\delta) \\ 0 \end{bmatrix}dB(t) - \textbf{a} F(t) \begin{bmatrix} -K(t-\delta) \\ 0 \end{bmatrix} dt.
\]

By integrating between times $0$ and $t$,

\begin{equation}
\label{eq: solution_SDE_exact}
 \begin{array}{lll}
  Y(t) &=& e^{\textbf{a} B(t) - \frac{1}{2} \textbf{a}^2 t} \Big( Y(0) + \int_0^t F(s) \begin{bmatrix} -K(s-\delta) \\ 0 \end{bmatrix} dB(s) \\[\smallskipamount]
  &&- \int_0^t \textbf{a} F(s) \begin{bmatrix} -K(s-\delta) \\ 0 \end{bmatrix} ds \Big)
 \end{array}
\end{equation}
\noindent where $Y(0) = (1, -B(-\delta))^{\intercal}$.

The first component of this solution $Y(t)$ is the exact solution $X(t)$ of the noisy memory SDE~\eqref{eq: num_ex_simple_SDE} for times $t \in [0, \delta]$. Furthermore, one can continue and iteratively solve \eqref{eq: num_ex_simple_SDE} for the interval $[\delta,2\delta]$ using the solution based on~\eqref{eq: solution_SDE_exact} as an initial condition. By continuing like this, one can solve the equation on the entire interval $[0,T]$. We will not calculate more solutions, as the one calculated above is sufficient for our goal of illustrating the Euler method.

We now compare the exact solution just derived to the numerical approximation based on the Euler method. Let $\delta=1$ and $T=1$. It would perhaps be more realistic to choose $T$ larger than $\delta$ (i.e., the time span of interest is greater than the time of memory). However, as the previous exact solution gets very complicated for $\delta < T$, we restrict ourselves to the case $\delta=T$. Figure \ref{fig: plot} shows $1000$ different simulations of the paths of the exact solution~\eqref{eq: solution_SDE_exact} have been plotted against the corresponding paths of the approximated Euler solution using time steps of size $\Delta t= 1/100$. In addition, the corresponding mean square error has been computed by Monte Carlo simulation (with these $1000$ simulations), and this error has also been plotted. As seen by the dashed line (representing the mean square error of the Euler approximation method) in Figure \ref{fig: plot}, the Euler method approximates the exact solution well in a mean square sense.

\begin{figure}[h!]
    \centering
    \includegraphics[width=1.1\textwidth]{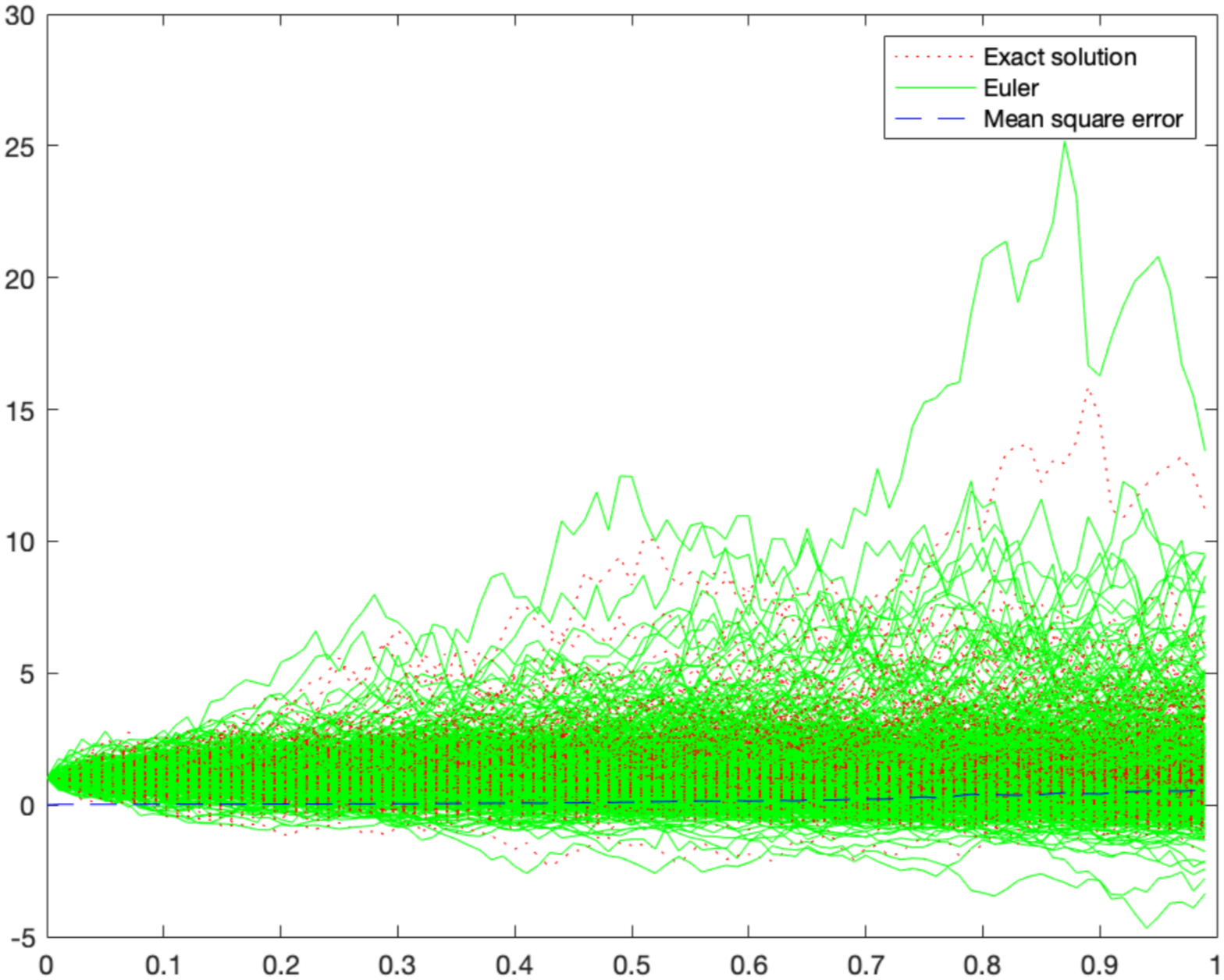}
    \caption{Plot of $1000$ paths of the exact solution of noisy SDE and the corresponding Euler approximation paths as well as the mean square error computed by Monte Carlo simulation.}
    \label{fig: plot}
\end{figure}


\begin{thebibliography}{99}

\bibitem{Allen}
Allen, E. (2007),  \newblock
Modeling with It{\^o} stochastic differential equations, \newblock
Springer, Dordrecht.

\bibitem{BerettaEtAl}
Berettaa, E., Kolmanovskiib, V.,  Shaikhetc, L. (1998),
Stability of epidemic model with time delays influenced by stochastic perturbations,
Mathematics and Computers in Simulation, 45: 269-277.



\bibitem{Buckwar}
Buckwar, E., (2006)\newblock
Introduction to the numerical analysis of stochastic delay differential equations, \newblock
Journal of Computational and Applied Mathematics, 125: 297-307.


\bibitem{Buckwar2}
Buckwar, E. (2004),  \newblock
The $\Theta$-Maruyama scheme for stochastic functional differential equations with distributed memory term, \newblock
Monte Carlo Methods and Applications, 10: 235-244.


\bibitem{Carletti}
Carletti, M. (2006), \newblock
Numerical solution of stochastic differential problems in the biosciences, \newblock
Journal of Computational and Applied Mathematics, 185: 422 - 440.


\bibitem{Dahl}
Dahl, K.R.,  Mohammed, S. E. A., {\O}ksendal, B. and R{\o}se, E. E. (2016),\newblock
Optimal control of systems with noisy memory and BSDEs with Malliavin derivatives, \newblock
Journal of Functional Analysis, 271: 289-329.

\bibitem{HighhamEtAl}
Highham, D. J., Mao, X., Stuart, A.M. (2002),\newblock
Strong convergence of Euler methods for non-linear stochastic differential equations, \newblock
SIAM Journal of Numerical Analysis, 40: 1041-1063.

\bibitem{Holden}
Holden, H., {\O}ksendal, B., Ub{\o}e, J.  and Zhang, T. (2010),\newblock
Stochastic Partial Differential Equations: A Modeling, White Noise Functional Approach, \newblock
Springer, Berlin Heidelberg.

\bibitem{Iacus}
Iacus, S.M. (2008),\newblock
Simulation and Inference for Stochastic Differential Equations, \newblock
Springer, New York.

\bibitem{Kolmanovskii}
Kolmanovskii, V. and Myshkis, A. (1999), 
\emph{Introduction to the Theory and Applications of Functional Differential Equations},
Springer, Dordrecht.

\bibitem{Kuang}
Kuang, Y., ed. (1993), Delay differential equations: with applications in population dynamics. Vol. 191. Academic Press.

\bibitem{Li}
Li, X. and Cao, W. (2015),
On mean-square stability of two-step Naruyama methods for nonlinear neural stochastic delay differential equations,
Applied Mathematics and Computation.

\bibitem{Mao}
Mao, X. (2007), \newblock
Stochastic Differential Equations and Applications, \newblock
second edition, Horwood Publishing, Chichester UK.

\bibitem{MaoSabanis}
Mao, X. and Sabanis, S. (2003),\newblock
Numerical solutions of stochastic differential delay equations under local Lipschitz condition, \newblock 
Journal of Computational and Applied Mathematics, 151: 215-227.

\bibitem{Milstein}
Milstein, G. N. and Tretyakov, M.V. (2004),
\emph{Stochastic Numerics for Mathematical Physics},
Springer, Berlin Heidelberg.

\bibitem{MR754561}
Mohammed, S.~E.~A. (1984),\newblock 
Stochastic functional differential equations, \newblock 
Research Notes in Mathematics, 99, Pitman (Advanced Publishing Program), Boston.

\bibitem{Mohammed_Memory}
Mohammed, S.~E.~A. (1996),\newblock
Stochastic Differential Systems with Memory: Theory, Examples and Applications, \newblock
in L. Decreusefond, J. Gjerde, B. Oksendal, S. Ustunel (Eds.), Stochastic Analysis and Related Topics VI: Proceedings of the Sixth Oslo-Silivri Workshop, Geilo.

\bibitem{MohammedZhang}
Mohammed, S.~E.~A. and  Zhang, T. (2009),\newblock
Anticipating stochastic differential systems with memory, \newblock
Stochastic Processes and their Applications, 119: 2773-2802.

\bibitem{Oksendal}
{\O}ksendal, B. (2007),\newblock
Stochastic differential equations, \newblock
sixth edition, fourth printing, Springer, Berlin Heidelberg.

\bibitem{OksendalSulem}
{\O}ksendal, B. and Sulem, A. (2000),
A maximum principle for optimal control of stochastic systems with delay with applications to finance,
\emph{Optimal Control and Partial Differential Equations - Innovations and Applications}, editors: J.M. Menaldi, E. Rofman and A. Sulem, IOS Press, Amsterdam.

\bibitem{OSZ1}
{\O}ksendal, B., Sulem, A. and Zhang, T. (2011),\newblock
Optimal control of stochastic delay equations and time-advanced backward stochastic differential equations,\newblock
Advances in Applied Probability, 43: 572-596.

\bibitem{OksendalZhang_Volterra}
{\O}ksendal, B. and Zhang, T. (1993),\newblock
The stochastic Volterra equation, \newblock
in D.~Nualart, M.~S.~Sol{\'e} (eds): Barcelona seminar on stochastic analysis, Birkh{\"a}user: 168-202.

\bibitem{OksendalZhang_linear}
{\O}ksendal, B. and  Zhang, T. (1996),\newblock
The general linear stochastic Volterra equation with anticipating coefficients, \newblock
Stochastic analysis and applications, World scientific Publishing: 343-366.

\bibitem{Verriest}
Verriest, E. I. and Florchinger, P. (1995),
Stability of stochastic systems with uncertain time delays,
Systems and Control Letters, 24: 41-47.


\bibitem{Verriest2}
Verriest, E. I. (2002), 
Stability of systems with state-dependent and random delays,
IMA Journal of Mathematics Control and Information, 19: 103-114.

\bibitem{VerriestMichiels}
Verriest, E. I. and Michiels, W. (2009),
Stability analysis of systems with stochastically varying delays,
Systems and Control Letters, 58: 783-791.



\bibitem{Wang}
Wang, Z. (2008),\newblock
Existence-uniqueness of solutions to stochastic Volterra equations with singular kernels and non-Lipschitz coefficients,
Statistics and Probability Letters, 78: 1062-1071.

\end{thebibliography}
\end{document}